\documentclass{article}%
\usepackage{amsfonts}
\usepackage{amssymb}
\usepackage{amsmath}
\usepackage{amsthm}
\usepackage{fullpage}
\usepackage{xcolor}
\usepackage{graphicx}
\usepackage{graphics}
\usepackage{pstricks}
\usepackage{pst-node}
\usepackage{pst-plot}
\usepackage{slashbox}
\usepackage{float}%
\providecommand{\U}[1]{\protect\rule{.1in}{.1in}}
\allowdisplaybreaks
\newtheorem{theorem}{Theorem}
\newtheorem{lemma}{Lemma}
\newtheorem{proposition}[lemma]{Proposition}

\newtheorem{example}{Example}
\newtheorem{remark}{Remark}

\newtheorem{definition}{Definition}

\begin{document}

\title{A Generalization of Jacobsthal and Jacobsthal-Lucas numbers }

\author{Alaa Al-Kateeb\footnote{
Departement of mathematics, Yarmouk university, Jordan.}}

\maketitle

\begin{abstract}In this paper, we study a  generalization of Jacobsthal and Jacobsthal-Lucas numbers, we find their generating function binet formulas, related matrix representation and many other properties. 
\end{abstract}

\section{Introduction}
Fibonacci and Lucas  integer  sequences and their generalization/ extensions  have many interesting, pretty and amazing properties and applications in many  fields of science and arts \cite{A0, A1, A2,A3,A4,A5,A6}.
The Fibonacci / Lucas sequences are given by the following recurrence relations
\[F_{n}= F_{n-1}+F_{n-2}, L_{n}= L_{n-1}+L_{n-2}\]
where $ n \geq 2, F_0=0, F_1=1, L_0=2$ and $L_{1}=1$.  
There are other   Fibonacci and Lucas  type sequences such as (\cite{wiki2,wiki})  \begin{itemize}
\item Pell and Pell-Lucas numbers: 
\[P_{n}= 2P_{n-1}+P_{n-2}, Q_{n}= 2Q_{n-1}+Q_{n-2}\]
where $n \geq 2,P_0=0, P_1=1, Q_0=Q_{1}=1 $
\item Jacobsthal and Jacobsthal-Lucas numbers \[J_{n}= J_{n-1}+2J_{n-2}, j_{n}= j_{n-1}+2j_{n-2}\]
where $n \geq 2,J_0=0, J_1=1, j_0=j_{1}=2 $
\end{itemize}
All listed  above sequences satisfy  a set of   common properties and  identities. \\
 Let  be an integer $k \geq 1$. then we can define the $k-$ Fibonacci  numbers by $F_{k,n}=kF_{k,n-1}+F_{k,n-2}$ for $n \geq 2,F_{k,0}=0$ and $F_{k,1}=1$.  We can  define $k-$ Lucas, Pell etc numbers in the same way see (\cite{FP,CA1,CA2, JK}).  \\

In \cite{SW} a new one-parameter generalization of Pell and Pell-Lucas numbers  numbers is introduced and  its properties and  related  matrix representation is studied, the generalization is given by  \[P_{k,n}= kP_{k,n-1}+(k-1)P_{k,n-2}, Q_{k,n}= kQ_{k,n-1}+Q_{k,n-2}\]
where $k,n \geq 2,P_{k,0}=0, P_{k,1}=1, Q_{k,0}=Q_{k,1}=2 $
\\

 In this paper we introduce and study  a similar  generalization of the  Jacobsthal and Jacobsthal-Lucas numbers.
This paper is structured as follows  in section \ref{sec:def} we introduce the generalized Jacobsthal and Jacobsthal-Lucus numbers and  derive their generating functions and binet formulas, in section \ref{sec:matrix} we find the matrices of the generalized Jacobsthal and Jacobsthal-Lucas numbers, in section \ref{sec:pro} we find many properties of the sequences  like the Cassini, Catalan and  d'Ocagne's formulas  finally in section \ref{sec:sum} we find the sum of terms formulas of the generalized
sequences. 
\section{Definition, generating functions and Binet formulas }\label{sec:def} 
\begin{definition}\label{def:gen}Let $k \geq 2, n\geq 0$ be two integers. We  define the generalized Jacobsthal and Jacobsthal-Lucas numbers respectively by 
\[J_{k,n}=(k-1)J_{k,n-1}+kJ_{k,n-2},  ~~ j_{k,n}=(k-1)j_{k,n-1}+kj_{k,n-2}\] 
where  $J_{k,0}=0, J_{k,1}=1$ and $ j_{k,0}=j_{k,1}=2$.
\end{definition} \noindent
We see that when $k=2$ we have the classical Jacobsthal and Jacobsthal-Lucas numbers.
\begin{example} In the following two tables  we present the values of  $J_{k,n}$ and $j_{k,n}$ for some selected values of $k$.
\begin {table}[H]
\caption { Generalized Jacobsthal numbers} \label{tab:1}\begin{center}\begin{tabular}{|c|c|c|c|c|c|c|c|c|c|c|c|}\hline
$n$ & 0 & 1 & 2 & 3 & 4 & 5 & 6 & 7 & 8&9&10 \\\hline
$J_{2,n}$ & 0 & 1 & 1 & 3 & 5 & 11 & 21 & 43 & 85&171& 341\\\hline
$J_{3,n}$ & 0 & 1 & 2 & 7 & 20 & 61 & 182 & 547 & 1640 &4921&14762\\\hline
$J_{4,n}$ & 0 & 1 & 3 & 13 & 51 & 205 & 819 &3277 & 13107 &52429&209715\\\hline
\end{tabular}
\end{center}
\end{table}
\begin {table}[H]
\caption { Generalized Jacobsthal-Lucas numbers} \label{tab:2}\begin{center}\begin{tabular}{|c|c|c|c|c|c|c|c|c|c|c|c|}\hline
$n$ & 0 & 1 & 2 & 3 & 4 & 5 & 6 & 7 & 8 &9&10\\\hline
$j_{2,n}$ & 2 & 2 & 6 & 10 & 22 & 42 & 86 & 170 & 342&682&1366 \\\hline
$j_{3,n}$ & 2 & 2 & 10 & 26 & 82 & 242 & 730 & 2186 & 6562 &19682&59050\\\hline
$j_{4,n}$ & 2 & 2 & 14& 50 & 206 & 818 & 3278 & 13106 & 52430 &209714&838862\\\hline
\end{tabular}
\end{center}
\end{table}

\end{example}
In the following theorem we derive the generation functions of the sequences $J_{k,n}$ and $j_{k,n}$.
\begin{theorem}[Generating functions]The generating functions of the sequences $J_{k,n}$ and $j_{k,n}$ respectively are \begin{enumerate}
\item \[J(x)=\frac{x}{1-(k-1)x-kx^2}\]
\item \[j(x)=\frac{2(x+2-k)}{1-(k-1)x-kx^2}\]
\end{enumerate}
\end{theorem}
\begin{proof} Let $J(x)$ and $j(x)$ represents the generating functions of $J_{k,n}$ and $j_{k,n}$ respectively.\\
Note, 
\begin{align*}J(x)&= \sum_{n=0}^{\infty} J_{k,n} x^n\\
                  &=J_{k,0}+J_{k,1}x+\sum_{n=2}^{\infty} J_{k,n} x^n\\
                  &=x+\sum_{n=2}^{\infty} ((k-1)J_{k,n-1}+kJ_{k,n-2}) x^n\\
                  &= x+(k-1)x\sum_{n=2}^{\infty} J_{k,n-1}x^{n-1}+kx^{2}\sum_{n=2}^{\infty} J_{k,n-2} x^{n-2}\\
                  &=x+(k-1)x\sum_{n=0}^{\infty} J_{k,n}x^{n}+kx^{2}\sum_{n=0}^{\infty} J_{k,n} x^{n}& J_{k,0}=0\\
&=x+(k-1)x J(x)+kx^2J(x)\\
\end{align*}
Thus, $x=(1+(1-k)x-kx^2)J(x)$ and $J(x)=\frac{x}{1-(k-1)x-kx^2}$.\\
Similarly,
\begin{align*}j(x)&= \sum_{n=0}^{\infty} j_{k,n} x^n\\
                  &=2+2x+\sum_{n=2}^{\infty}((k-1)j_{k,n-1}+kj_{k,n-2}) x^n\\
                  &=2+2x+(k-1)x\sum_{n=2}^{\infty}j_{k,n-1}x^{n-1}+kx^2 \sum_{n=2}^{\infty}j_{k,n-2} x^{n-2}\\
                  &=4-2k+2x+(k-1)x\sum_{n=0}^{\infty}j_{k,n}x^{n}+kx^2 \sum_{n=0}^{\infty}j_{k,n} x^{n}& \text{by adding and subtracting}  ~~~2(k-1)\\
\end{align*}
Thus, $4-2k+2x=(1-(k-1)x-kx^2)j(x)$ and $j(x)=\frac{2x+4-2k}{1-(k-1)x-kx^2}$.
 \end{proof}
\begin{theorem}[Binet formulas]\label{thm:binet} The $n$th terms of the generalized Jacobsthal and Jacobsthal-Lucas sequences are given by
\begin{equation}\label{equ:binet1}J_{k,n}=\frac{k^n-(-1)^{n}}{k+1}, \end{equation}
\begin{equation}\label{equ:binet2}j_{k,n}=\frac{4k^n+2(k-1)(-1)^n}{k+1}\end{equation}

\end{theorem}
\begin{proof}We will use the mathematical induction to prove the formulas. To prove equation \ref{equ:binet1}.
\begin{itemize}
\item n= 0: $J_{k,0}=\frac{1-1}{k+1}=0$.
\item Assume that $J_{k,n}=\frac{k^n-(-1)^{n}}{k+1}$
\item Note
\begin{align*}J_{k,n+1}&= (k-1)J_{k,n}+kJ_{k,n-1}\\
                     &=(k-1)(\frac{k^n-(-1)^{n}}{k+1})+k(\frac{k^{n-1}-(-1)^{n-1}}{k+1})\\
                     &=\frac{k^{n+1}-k(-1)^{n}-k^{n}+(-1)^{n}}{k+1}+\frac{k^{n}-k(-1)^{n-1}}{k+1}\\
&= \frac{k^{n+1}-(-1)^{n+1}}{k+1}
\end{align*}
\end{itemize}
 Similarly,we  prove equation \ref{equ:binet2}.
\begin{itemize}
\item n= 0: $j_{k,0}=\frac{4k^0+(2k-2)(-1)^0}{k+1}=2$.
\item Assume that $j_{k,n}=\frac{4k^n+(2k-2)(-1)^n}{k+1}$
\item  Note
\begin{align*}j_{k,n+1}&= (k-1)j_{k,n}+kj_{k,n-1}\\
                     &=(k-1)(\frac{4k^n+(2k-2)(-1)^n}{k+1})+k(\frac{4k^{n-1}+(2k-2)(-1)^{n-1}}{k+1})\\
                     &=\frac{4k^{n+1}+(2k-2)(-1)^{n+1}}{k+1}\\
\end{align*}
\end{itemize}as desired
 \end{proof}

\section{Matrix representation}\label{sec:matrix}
Classical Jacobsthal   numbers can be derived from the  matrix $F=\begin{bmatrix}1 & 2 \\
1 & 0 \\
\end{bmatrix}$   for which $F^n=\begin{bmatrix}J_{n+1} & 2J_n \\
J_n & 2J_{n-1} \\
\end{bmatrix}$ (the Jacobsthal F-matrix ). 
Also the Jacobsthal-Lucas numbers can be derived from the matrix $R=\begin{bmatrix}1 & 4 \\
2 & -1 \\
\end{bmatrix}$ (the Jacobsthal-Lucas R-matrix ) for which we can define $R_n=RF^n= \begin{bmatrix}j_{n+1} & 2j_{n} \\
j_n & 2j_{n-1} \\
\end{bmatrix}$, the matrices were studied in details in \cite{BO1,BO2}. We observe that  the matrices $F$ and $R$ are related to each other in a similar way like the Fibonacci Q-matrix and the Lucas R-matrix.
 In this section we derive the Jacobsthal $F_k$  and the Jacobsthal-Lucas $R_k$ matrices for the generalized Jacobsthal/Jacobsthal-Lucas numbers respectively.
\begin{lemma}[Matrix of generalized Jacobsthal numbers]\label{lem:F}
Let $F_k=\begin{bmatrix}k-1 & k \\
1 & 0 \\
\end{bmatrix}$. Then   for $n \geq 2$ we have \[F_k^n=\begin{bmatrix}J_{k,n+1} & kJ_{k,n} \\
J_{k,n} &k J_{k,n-1} \\
\end{bmatrix}\]
\end{lemma}
\begin{proof} We prove the Lemma mathematical induction:
\begin{itemize}
\item  For $n=2$: 
\[F_k^2=\begin{bmatrix}(k-1)^2+k & k \\
k-1 & k \\
\end{bmatrix}=\begin{bmatrix}J_{k,2+1} & kJ_{k,2} \\
J_{k,2} &k J_{k,1} \\
\end{bmatrix}\]
\item Assume that $F_k^n=\begin{bmatrix}J_{k,n+1} & kJ_{k,n} \\
J_{k,n} & kJ_{k,n-1} \\
\end{bmatrix}$
\item Note
\begin{align*}F_k^{n+1}&=F_k F_k^n\\
&= F_k\begin{bmatrix}J_{k,n+1} & kJ_{k,n} \\
J_{k,n} & kJ_{k,n-1} \\
\end{bmatrix} & \text{by induction}\\
&= \begin{bmatrix}k-1 & k \\
1 & 0 \\
\end{bmatrix} \begin{bmatrix}J_{k,n+1} & kJ_{k,n} \\
J_{k,n} & kJ_{k,n-1} \\
\end{bmatrix}\\
&= \begin{bmatrix}(k-1)J_{k,n+1}+kJ_{k,n} &k(k-1) J_{k,n}+k^2J_{k,n-1} \\
J_{k,n+1} & kJ_{k,n} \end{bmatrix} \\
&= \begin{bmatrix}J_{k,n+2} & kJ_{k,n+1} \\
J_{k,n+1} & kJ_{k,n} \\
\end{bmatrix}  
\end{align*}
\end{itemize} 
as desired
\end{proof}
\begin{lemma}For $n \geq 1, \begin{bmatrix}j_{k,n+1} \\
j_{k,n} \\
\end{bmatrix}=F_k\begin{bmatrix}j_{n} \\
j_{n-1} \\
\end{bmatrix}$
\end{lemma}
\begin{proof}Immediate \end{proof}
\begin{proposition}\label{pro:interms}For $n \geq 1$ we have \begin{enumerate}
\item $j_{k,n} =2(J_{k,n}+kJ_{k,n-1})
$
\item $j_{k,n-1} =2(J_{k,n}+(2-k)J_{k,n-1})
$
\end{enumerate} 
\end{proposition}
\begin{proof}Note 
\begin{align*}
J_{k,n}+kJ_{k,n-1} &= \frac{k^n-(-1)^{n}}{k+1}+k\frac{k^{n-1}-(-1)^{n-1}}{k+1}\\
&= \frac{2k^n-k(-1)^{n-1}-(-1)^{n}}{k+1}\\
&= \frac{2k^n+(k-1)(-1)^{n}}{k+1}\\
&= \frac{1}{2} j_{k,n}
\end{align*} \end{proof}
\begin{lemma}[Matrix of generalized Jacobsthal-Lucas  numbers ]\label{lem:R}
Let $R_k=\begin{bmatrix}1 & k \\
1 & 2-k \\
\end{bmatrix}$. Then \[R_kF_k^n= \frac{1}{2}\begin{bmatrix}j_{k,n+1} & kj_{k,n} \\
j_{k,n} & kj_{k,n-1} \\
\end{bmatrix}\]
\end{lemma}
\begin{proof}
\begin{align*}R_kF_k^n&=R_k\begin{bmatrix}J_{k,n+1} & kJ_{k,n} \\
J_{k,n} &k J_{k,n-1} \\
\end{bmatrix}& \text{from Lemma} ~ \ref{lem:F}\\
&= \begin{bmatrix}1 & k \\
1 & 2-k \\
\end{bmatrix} \begin{bmatrix}J_{k,n+1} & kJ_{k,n} \\
J_{k,n} &k J_{k,n-1} \\
\end{bmatrix}\\&= \begin{bmatrix}J_{k,n+1} + kJ_{k,n} & k^2J_{k,n-1} 
+ kJ_{k,n} \\
J_{k,n+1} +(2-k)J_{k,n} & kJ_{k,n} +(2-k)kJ_{k,n-1} \\
\end{bmatrix}
\\&= \frac{1}{2}\begin{bmatrix}j_{k,n+1} & kj_{k,n} \\
j_{k,n} & kj_{k,n-1} \\
\end{bmatrix} & \text{from Proposition} ~\ref{pro:interms}
\end{align*}
\end{proof}
\begin{remark} $R_kF_k=F_kR_k$
\end{remark}
\section{More properties and identities }\label{sec:pro}
In this section, we give  identities for
the generalized  Jacobsthal  and Jacobsthal-Lucas sequences.
\begin{theorem}[Catalan's Identities]
\end{theorem}
\begin{proof}\
\begin{enumerate}
\item For simplicity let us write 
\begin{align*}J_{k,n+r}J_{k,n-r}-J_{k,n}^2&= \frac{k^{n+r}+(-1)^{n+r}}{k+1}\cdot \frac{k^{n-r}+(-1)^{n-r}}{k+1}-\left(\frac{k^n+(-1)^{n}}{k+1}\right)^2\\
&=\frac{k^{2n}+(-1)^{n+r}k^{n-r}+(-1)^{n-r}k^{n+r}+1}{(k+1)^2} -\left(\frac{k^n+(-1)^{n}}{k+1}\right)^2\\ & =\frac{(-1)^{n+r}k^{n-r}+(-1)^{n-r}k^{n+r}+2(-1)^{n}k^n}{(k+1)^2}\\
 &=(-1)^{n}k^n\left(\frac{2+(-1)^{r}k^{-r}+(-1)^{r}k^{r}}{(k+1)^2}\right)\\
&= \frac{(-1)^{n-r}k^n}{k^r}\left(\frac{2(-1)^r(k)^r+1+k^{2r})}{(k+1)^2}\right)\\ &= (-1)^{n-r}k^{n-r} J_{k,r}^2   
\end{align*}
\item  Note
\begin{align*} j_{k,n+r}j_{k,n-r}-j_{k,n}^2&=\frac{4k^{n+r}+2(k-1)(-1)^{n+r}}{k+1} \cdot \frac{4k^{n-r}+2(k-1)(-1)^{n-r}}{k+1} -j_{k,n}^2\\
&= \frac{16k^{2n}+8k^{n-r}(k-1)(-1)^{n+r}+8k^{n+r}(k-1)(-1)^{n-r}+4(k-1)^2}{(k+1)^{2}} -j_{k,n}^2\\
&= \frac{8k^{n-r}(k-1)(-1)^{n+r}+8k^{n+r}(k-1)(-1)^{n-r}-16(k-1)k^{n}(-1)^n}{(k+1)^2} \\
&=8(-1)^nk^n(k-1) \cdot  \frac{(k^{-r}(-1)^{r}+k^{r}(-1)^{-r}-2)}{(k+1)^2} \\
&=8(-1)^nk^n(k-1) \cdot  \frac{(-2+k^{-r}(-1)^{r}+k^{r}(-1)^{-r})}{(k+1)^2} \\
&=8(-1)^{n-r}k^n(k-1) \cdot  \frac{k^{2r}-2(-1)^rk^r+1}{(k+1)^2} \\
&=8(-1)^{n-r}k^n(k-1) J_{k,r}^2 \end{align*}
\end{enumerate}
\end{proof}
\begin{theorem}[Cassini's identities] For $n \ge2 $ we have
\begin{enumerate}
\item $J_{k,n+1}J_{k,n-1}-J_{k,n}^2= (-1)^nk^{n-1}$
\item $j_{k,n+1}j_{k,n-1}-j_{k,n}^2= 8(-1)^nk^{n-1} (1-k)$
\end{enumerate}
\end{theorem}
\begin{proof} Immediate from Lemma \ref{lem:F} and Lemma  \ref{lem:R}
\end{proof}
\begin{theorem}[d'Ocagne's Identity identities] Let $n \geq m$ be two integers.
 Then\begin{enumerate}
\item $J_{k,n}J_{k,m+1}-J_{k,n+1}J_{k,m}= (-1)^mk^m J_{k,n-m}$
\item $j_{k,n}j_{k,m+1}-j_{k,n+1}j_{k,m}= 8(-1)^m(1-k)k^m J_{k,n-m}$
\end{enumerate} \end{theorem}
\begin{proof} First using equation \ref{equ:binet1} in Theorem \ref{thm:binet} we have 
\begin{align*}J_{k,n}J_{k,m+1}-J_{k,n+1}J_{k,m}&=\frac{k^n-(-1)^{n}}{k+1} \frac{k^{m+1}-(-1)^{m+1}}{k+1}-\frac{k^m-(-1)^{m}}{k+1} \frac{k^{n+1}-(-1)^{n+1}}{k+1}\\
&=\frac{k^{n+m+1}-(-1)^{n}k^{m+1}-(-1)^{m+1}k^{n}+(-1)^{m+n+1}}{(k+1)^2}\\&  -\frac{k^{m+n+1}-(-1)^{n+1}k^m-(-1)^{m}k^{n+1}+(-1)^{m+n+1}}{(k+1)^2}\\
&= \frac{(-1)^{n+1}k^m-(-1)^{n}k^{m+1}+(-1)^{m}k^{n+1}-(-1)^{m+1}k^{n}}{(k+1)^2}\\
&= \frac{(-1)^{n}k^m(-1-k)+(-1)^{m}k^n(k+1)}{(k+1)^2}\\ 
&= \frac{(-1)^{m}k^n-(-1)^{n}k^m}{(k+1)}\\
&= (-1)^mk^m\left(\frac{k^{n-m}-(-1)^{n-m}}{(k+1)}\right)\\
&= (-1)^mk^m J_{k,n-m}\end{align*}
 Second  using equation \ref{equ:binet2} in Theorem \ref{thm:binet} we have \begin{align*}j_{k,n}j_{k,m+1}-j_{k,n+1}j_{k,m}&=\frac{8k^n(-1)^{m}+8k^{m+2}(-1)^{n}-8k^m(-1)^{n}-8k^{n+2}(-1)^m}{(k+1)^2}\\
&= 8 \frac{(-1)^mk^n(1-k^2)+(-1)^nk^m(k^2-1)}{(k+1)^2}\\
&= 8(1-k^2) \frac{(-1)^mk^n+(-1)^nk^m}{(k+1)^2}\\
&= 8(1-k^2)(-1)^mk^m \frac{k^{n-m}+(-1)^{n-m}}{(k+1)^2}\\
&= 8(-1)^m(1-k)k^m J_{k,n-m}
\end{align*}
\end{proof}
\begin{theorem} For any  two integers $m,n \geq 2$ we have
\begin{enumerate}
\item $J_{k,m+n}=J_{k,m}J_{k,n+1}+kJ_{k,m-1}J_{k,n-1}$
\item  $j_{k,m+n}=j_{k,m}J_{k,n+1}+kj_{k,m-1}J_{k,n}$
\end{enumerate}\end{theorem}
\begin{proof}\
\begin{enumerate}
\item \[F_k^{m+n}=\begin{bmatrix}J_{k,m+n+1} & kJ_{k,m+n} \\
J_{k,m+n} &k J_{k,m+n-1} \\
\end{bmatrix}=\begin{bmatrix}J_{k,m+1} & kJ_{k,m} \\
J_{k,m} &k J_{k,m-1} \\
\end{bmatrix}\begin{bmatrix}J_{k,n+1} & kJ_{k,n} \\
J_{k,n} &k J_{k,n-1} \\
\end{bmatrix}\]
thus $J_{k,m+n}=J_{k,m}J_{k,n+1}+kJ_{k,m-1}J_{k,n-1}$
\item Similar to number 1

\end{enumerate}
\end{proof}
\section{Sum of terms}\label{sec:sum}
\begin{theorem} For all integers $k \geq 2$ and $n\geq 0$ we have 
\begin{enumerate}
\item $\sum_{i=0}^n J_{k,i}= \frac{1}{2(k-1)} (kJ_{k,n}+J_{k,n+1}-1)$
\item $\sum_{i=0}^n j_{k,i}= \frac{1}{2(k-1)} (kj_{k,n}+j_{k,n+1}+2(k-3))$
\end{enumerate} 
\end{theorem}
\begin{proof}\
We prove the first formula using mathematical induction
\begin{enumerate}
\item $n=0$ the result is trivial.
\item Assume that $\sum_{i=0}^n J_{k,i}= \frac{1}{2(k-1)} (kJ_{k,n}+J_{k,n+1}-1)$
\item Consider 
\begin{align*}\sum_{i=0}^{n+1} J_{k,i}& =\sum_{i=0}^{n} J_{k,i}+ J_{k,n+1}\\
&=  \frac{1}{2(k-1)} (kJ_{k,n}+J_{k,n+1}-1)+ J_{k,n+1} & \text{by induction}\\
&=  \frac{1}{2(k-1)} (kJ_{k,n}+J_{k,n+1}+2(k-1)J_{k,n+1}-1)\\
&=  \frac{1}{2(k-1)} (kJ_{k,n+1}+(k-1)J_{k,n+1}+kJ_{k,n}-1)\\
&=  \frac{1}{2(k-1)} (kJ_{k,n+1}+J_{k,n+2}-1)
\end{align*}
\end{enumerate}
The second formula can be proved in a similar way to the first one so we omit its proof.
\end{proof}
\bibliographystyle{unsrt}
\bibliography{all}
{}

\end{document}